\documentclass[a4paper,12pt]{amsart}
\usepackage{amsmath,amsthm,amssymb,mathtools,mathrsfs}
\usepackage{bm}
\usepackage{tikz-cd}
\usepackage{xcolor}
\usepackage{hyperref}
    \definecolor{urlcolor}{rgb}{0,0,0}
    \definecolor{linkcolor}{rgb}{.7,0.10,0.2}
    \definecolor{citecolor}{rgb}{.12,.54,.11}
\hypersetup{
      breaklinks=true, 
      colorlinks=true,
      urlcolor=urlcolor,
      linkcolor=linkcolor,
      citecolor=citecolor,
}
\usepackage{cleveref}
\usepackage{setspace}
\usepackage{orcidlink}

\usepackage[shortlabels]{enumitem}
\usepackage[utf8]{inputenc}
\numberwithin{equation}{section}
\usepackage[top=4cm,
bottom=4cm,
left=2.5cm,
right=2.5cm]{geometry}

\usepackage[
style=alphabetic,
doi=false,
isbn=false,
url=false,
eprint=false,
block = space,
maxcitenames=100,
maxbibnames=100,
sorting=anyt]{biblatex} 

\newtheorem{theorem}{Theorem}[section]
\newtheorem{corollary}[theorem]{Corollary}
\newtheorem{proposition}[theorem]{Proposition}
\newtheorem{proposition-definition}[theorem]{Proposition-Definition}
\newtheorem{lemma}[theorem]{Lemma}

\theoremstyle{definition}

\newtheorem{theorem-definition}[theorem]{Theorem-Definition}

\theoremstyle{remark}

\renewcommand{\geq}{\geqslant}
\renewcommand{\leq}{\leqslant}
\renewcommand{\subset}{\subseteq}

\newcommand{\BD}{\mathbb{D}}

\newcommand{\BoA}{\mathbf{A}}
\newcommand{\BoC}{\mathbf{C}}

\newcommand{\BoN}{\mathbf{N}}

\newcommand{\BoQ}{\mathbf{Q}}

\newcommand{\BoZ}{\mathbf{Z}}

\newcommand{\CA}{\mathcal{A}}

\newcommand{\CC}{\mathcal{C}}
\newcommand{\CD}{\mathcal{D}}

\newcommand{\CF}{\mathcal{F}}
\newcommand{\CG}{\mathcal{G}}

\newcommand{\CM}{\mathcal{M}}

\newcommand{\SA}{\mathscr{A}}

\newcommand{\BPS}{\mathcal{BPS}}

\newcommand{\rmb}{\mathrm{b}}
\newcommand{\rmc}{\mathrm{c}}

\newcommand{\cms}{/\!\!/}

\newcommand{\diag}{\mathrm{diag}}

\newcommand{\Fg}{\mathfrak{g}}
\newcommand{\GL}{\mathrm{GL}}

\newcommand{\Frac}{\mathrm{Frac}}

\newcommand{\Gr}{\mathrm{Gr}}
\newcommand{\HO}{\mathrm{H}}
\newcommand{\Hom}{\mathrm{Hom}}

\newcommand{\id}{\mathrm{id}}

\newcommand{\JH}{\mathtt{JH}}

\newcommand{\nil}{\mathrm{nil}}

\newcommand{\Perv}{\mathrm{Perv}}

\newcommand{\pt}{\mathrm{pt}}

\newcommand{\rmBM}{\mathrm{BM}}

\newcommand{\Prim}{\mathrm{Prim}}

\newcommand{\Span}{\mathrm{Span}}

\newcommand{\Sym}{\mathrm{Sym}}

\newcommand{\SN}{\mathcal{SN}}

\title{Commutativity of nilpotent cohomological Hall algebras of $\BoA^2$}
\date{\today}

\author{Lucien Hennecart}

\address{Laboratoire Ami\'enois de Math\'ematique Fondamentale et Appliqu\'ee, CNRS UMR 7352, Universit\'e de Picardie Jules Verne, 33 rue Saint Leu, 80000 Amiens, France}\email{lucien.hennecart@u-picardie.fr}

\begin{document}

\begin{abstract}
In this paper, we prove that both the seminilpotent and the fully nilpotent cohomological Hall algebras (CoHAs) of $\mathbf{A}^2$ are commutative. This result is in strong contrast with the CoHA of $\BoA^2$ without nilpotency conditions, previously studied by Davison, which is related to the Lie algebra $W_{1+\infty}$ of differential operators on $\BoC^*$. The latter is highly noncommutative. Our proof combines two constraints on the Lie bracket on the affinized BPS Lie algebra of a quiver: it is filtered with respect to the perverse filtration and it is graded with respect to the cohomological degree. In the case of the Jordan quiver and nilpotent CoHAs, these constraints force the Lie bracket to vanish. We also describe the equivariant nilpotent CoHAs in the presence of the action of a one-dimensional torus rescaling the first coordinate of $\BoA^2$ with weight $1$ and the second with weight $-1$. In this case, one obtains enveloping algebras of Rees Lie algebras associated with the nilpotent and the seminilpotent filtrations on the Lie algebra $W_{1+\infty}^+$, reminiscent of the description of the equivariant non-nilpotent CoHA given by Davison.
\end{abstract}

\keywords{Geometric representation theory, cohomological Hall algebras, BPS Lie algebras, commuting variety}
\subjclass[2020]{Primary 16G20; Secondary 14N35, 14F43}
\maketitle

\setcounter{tocdepth}{1}
\tableofcontents

\section{Introduction}
\subsection{Background}
The cohomological Hall algebra (CoHA) of $\BoA^2$ (equivalently, of the preprojective algebra of the Jordan quiver) was first considered in \cite{schiffmann2013cherednik} in connection with the AGT conjectures. This is a particular case of CoHAs of preprojective algebras of quivers \cite{yang2018cohomological}. CoHAs have since become central objects, with applications in representation theory, enumerative geometry and algebraic geometry with descriptions of the cohomology of quiver varieties \cite{davison2023bps}, proof of the $\chi$-independence conjecture of Toda \cite{toda2023gopakumar} for symplectic surfaces \cite{davison2026hecke}, and proof of the P=W conjecture \cite{hausel2022p}, to mention only a partial and non-exhaustive list.

CoHAs of quivers with potentials were defined around the same time as the CoHA of $\BoA^2$, by Kontsevich and Soibelman in \cite{kontsevich2011cohomological}, in order to give a mathematical meaning to the physics expectation predicting the existence of an algebra of BPS states \cite{harvey1998algebras}. The CoHA of $\BoA^2$, and more generally the CoHA of the preprojective algebra of any quiver, may be viewed as a CoHA of a quiver with potential via dimensional reduction \cite[Appendix A]{davison2017critical}, \cite[Appendix by Ben Davison]{ren2017cohomological}, see also \cite{yang2020two}.

Despite their many applications, the computation of the algebraic structure of CoHAs remains a very challenging question. These algebras are fully described by generators and relations in only a few cases. One can mention the examples of finite type quivers, the Jordan quiver, and more generally cyclic quivers \cite{davison2022affine,jindal2024coha}.

CoHAs come in several variants given by imposing various flavours of nilpotent conditions (see \cite[\S1.1]{bozec2020number} for a list). In this paper, we are interested in the two possible nilpotent versions of the cohomological Hall algebra of $\BoA^2$, namely the fully nilpotent and the seminilpotent CoHA, corresponding respectively to finite length sheaves on $\BoA^2$ set-theoretically supported on the origin or on the horizontal axis of $\BoA^2$.

The CoHA of $\BoA^2$ without nilpotency conditions, $\CA_{\BoA^2}$, was studied by Davison in \cite{davison2022affine}. It is governed by a degeneration of the Lie algebra $W_{1+\infty}$ (\Cref{subsubsection:the_equivariant_deformation}). We briefly recall its structure. Let $\Gr^F_{\bullet} W_{1+\infty}^+\coloneqq\Span_{\BoQ}\{z^mD^n\colon m\geq 1, n\geq 0\}$ with the Lie bracket $[z^mD^a,z^{n}D^{b}]=(an-bm)z^{m+n}D^{a+b-1}$. The presence of the associated graded $\Gr^F_{\bullet}$ means that the bracket is the degeneration of a bracket on $W_{1+\infty}^+$ with respect to a filtration $F$, which we do not recall here. Then, one of the main results of \cite{davison2022affine} is that there is an isomorphism of graded algebras $\CA_{\BoA^2}\cong\mathbf{U}(\Gr^F_{\bullet}W_{1+\infty}^+)$.

The case of $\BoA^2$ is then a particular case of CoHAs of finite length sheaves on surfaces, which were studied first for cotangent bundles of curves by Minets in \cite{minets2020cohomological}, and then described in full generality by Mellit, Minets, Schiffmann and Vasserot in \cite{mellit2023coherent}.

The nilpotent CoHAs of finite length sheaves on surfaces have recently been studied in \cite{diaconescu2025nilpotent,diaconescu2026cohomological}.

Despite all these results, nilpotent versions of CoHAs are not yet fully understood, and this paper gives this description for $\BoA^2$, which happens to be simpler than the CoHA of $\BoA^2$ studied in \cite{davison2022affine}: they are commutative algebras.

\subsection{Main results}

\subsubsection{Commutativity of nilpotent cohomological Hall algebras of $\BoA^2$}
We let $\CA_{\BoA^2}$ be the cohomological Hall algebra (CoHA) of $\BoA^2$ (\Cref{section:CoHA-A2}), $\CA_{\BoA^2}^{\nil}$ the \emph{fully nilpotent CoHA} of $\BoA^2$ and $\CA_{\BoA^2}^{\SN}$ the \emph{seminilpotent CoHA} of $\BoA^2$ (\Cref{section:nilpotent-cohas}). The underlying vector spaces of these algebras are the Borel--Moore homologies
\[
 \CA_{\BoA^2}=\bigoplus_{d\in\BoN}\HO^{\rmBM}_*(\CC(\mathfrak{gl}_d),\BoQ),\hspace{0.2cm}
 \CA_{\BoA^2}^{\nil}=\bigoplus_{d\in\BoN}\HO^{\rmBM}_*(\CC^{\nil}(\mathfrak{gl}_d),\BoQ),\hspace{0.2cm} \CA_{\BoA^2}^{\SN}=\bigoplus_{d\in\BoN}\HO^{\rmBM}_*(\CC^{\SN}(\mathfrak{gl}_d),\BoQ)
\]
where $\CC(\mathfrak{gl}_d)$ is the commuting stack of $\mathfrak{gl}_d$ (pairs of commuting $d\times d$ matrices), $\CC^{\nil}(\mathfrak{gl}_d)$ the nilpotent commuting stack (pairs of commuting nilpotent $d\times d$ matrices) and $\CC^{\SN}(\mathfrak{gl}_d)$ the seminilpotent stack (pairs of commuting $d\times d$ matrices of which the second is nilpotent). Note that the customary shift on the cohomological degree that is necessary for general quivers to make the CoHA product graded with respect to the cohomological degree is not present here. This is because this shift is given by the symmetrized Euler form of $Q$ which vanishes when $Q$ is the Jordan quiver.

There is a factorization coproduct on $\CA_{\BoA^2}, \CA_{\BoA^2}^{\nil}$ and $\CA_{\BoA^2}^{\SN}$ making them connected cocommutative bialgebras, see \cite[\S6.2]{davison2026hecke} for a general definition and also \cite[\S9]{jindal2024coha}, \cite[\S4.1]{hennecart2026degenerations} for the case of quivers. By the graded Milnor--Moore theorem, there are identifications of algebras
\[
 \CA_{\BoA^2}\cong \mathbf{U}(\hat{\Fg}_{\BoA^2}), \quad\CA_{\BoA^2}^{\nil}\cong \mathbf{U}(\hat{\Fg}^{\nil}_{\BoA^2}),\quad\CA_{\BoA^2}^{\SN}\cong \mathbf{U}(\hat{\Fg}^{\SN}_{\BoA^2})
\]
where $\hat{\Fg}_{\BoA^2}\subset\CA_{\BoA^2}$, $\hat{\Fg}^{\nil}_{\BoA^2}\subset\CA_{\BoA^2}^{\nil}$ and $\hat{\Fg}^{\SN}_{\BoA^2}\subset\CA_{\BoA^2}^{\SN}$ denote the respective Lie algebras of primitive elements.

\begin{theorem}
\label{theorem:commutativity-fully-nilpotent}
 The fully nilpotent CoHA of $\BoA^2$ is commutative. There are isomorphisms
 \[
  \hat{\Fg}^{\nil}_{\BoA^2}\cong s\BoQ[s,u], \quad \CA_{\BoA^2}^{\nil}\cong \mathbf{U}(s\BoQ[s,u])\cong \Sym(s\BoQ[s,u])
 \]
respectively of abelian Lie algebras and commutative associative algebras, where the cohomological degrees of generators are given by $\deg(s^k)=2$ (for any $k\geq 1$) and $\deg(u)=2$ and the dimension degree of $s^ku^l$ is $k$ for any $k,l$, so that the (dimension,cohomological)-degree of $s^ku^l$ is $(k,2+2l)$.
\end{theorem}

As mentionned to us by Sasha Minets, the commutativity of the fully nilpotent CoHA of $\BoA^2$ may be deduced from \cite{mellit2023coherent}. Indeed, using the $\BoC^*$-equivariance of the sheafified CoHA $\SA_{\BoA^2}$, the fully nilpotent CoHA $\CA_{\BoA^2}^{\nil}$ coincides with the compactly supported CoHA $(C(\mathfrak{gl})\rightarrow \pt)_!\SA_{\BoA^2}$ considered in \cite{mellit2023coherent}. Then by \cite[Theorem~B]{mellit2023coherent}, the compactly supported CoHA is isomorphic to the subalgebra of the CoHA of $\mathbb{P}^2$ denoted $W_{\downarrow}^+(\BoA^2)$ in \cite{mellit2023coherent}. The latter is the subalgebra defined in \cite[Definition~5.1(a)]{mellit2023coherent}, involving only the cohomology class $[\pt]$ of the surface. Then, by \cite[Theorem~A]{mellit2023coherent}, we obtain the commutativity of $\CA_{\BoA^2}^{\nil}$ since if $\lambda=\mu=[\pt]$, then the cup-product of $\lambda$ and $\mu$ is $0$.

\begin{theorem}
\label{theorem:commutativity-semi-nilpotent}
The seminilpotent CoHA of $\BoA^2$ is commutative. There are isomorphisms
\[
 \hat{\Fg}^{\SN}_{\BoA^2}\cong s\BoQ[s,u], \quad \CA_{\BoA^2}^{\SN}\cong \mathbf{U}(s\BoQ[s,u])\cong\Sym(s\BoQ[s,u])
\]
respectively of abelian Lie algebras and commutative associative algebras, where the cohomological degrees of generators are given by $\deg(s^k)=0$ (for any $k\geq 1$) and $\deg(u)=2$ and the dimension degree of $s^ku^l$ is $k$ for any $k,l$, so that the (dimension,cohomological)-degree of $s^ku^l$ is $(k,2l)$.
\end{theorem}

This is in strong contrast with the non-nilpotent cohomological Hall algebra of $\BoA^2$, which is studied in \cite{davison2022affine}. In this case, the affinized BPS Lie algebra is related to a degeneration of the Lie algebra $W_{1+\infty}$, and it is non-commutative. Taking the nilpotent restrictions changes the cohomological degrees of the generators, and together with the perverse degree constraints on the Lie bracket, this forces the Lie brackets to vanish.

There are other variants of nilpotent cohomological Hall algebras associated to surfaces. Namely, in \cite{diaconescu2026cohomological}, the authors study in particular the CoHA of sheaves on resolutions of Kleinian singularities supported on the exceptional divisor, which is typically a singular (nodal) projective curve. These are expected to be non-commutative in general.

\subsubsection{The equivariant deformation}
\label{subsubsection:the_equivariant_deformation}
With some extra work, one can describe the nilpotent and seminilpotent affinized BPS Lie algebra and CoHA in the presence of the action of $\BoC^*$ rescaling the first coordinate of $\BoA^2$ with weight $1$ and the second coordinate with weight $-1$ (\Cref{subsection:equivariant_versions}). Both descriptions are in terms of the Rees Lie algebra of the Lie algebra $W_{1+\infty}^+$ described in \cite{davison2022affine}.

We let $W_{1+\infty}^+$ be the Lie algebra over $\BoQ$ having the basis $\{z^mD^n\colon m\in\BoZ_{\geq 1}, n\in \BoZ_{\geq 0}\}$ with Lie bracket given by
\[
 [z^mD^a,z^nD^b]=z^{m+n}((D+n)^aD^b-D^a(D+m)^b).
\]
It is $\BoZ_{\geq 1}$-graded, with the $n$th graded piece given by the subspace generated by $z^nD^a$ for $a\geq 0$. We define the \emph{nilpotent cohomological filtration} by
\[
 F_i^{\nil}\coloneqq \Span_{\BoQ}\left(z^mD^a\colon 0\leq a\leq \frac{i-2}{2}, m\geq 1\right)
\]
and the \emph{seminilpotent cohomological filtration} by
\[
  F_i^{\SN}\coloneqq \Span_{\BoQ}\left(z^mD^a\colon 0\leq a\leq \frac{i}{2}, m\geq 1\right).
\]
Then, the Lie bracket on $W_{1+\infty}^+$ respects both filtrations.

We denote by
\[
 \mathbf{R}_{F^{\nil}}[W_{1+\infty}^+]\coloneqq \bigoplus_{i\geq 1}F_{2i}^{\nil}t^i\subset W_{1+\infty}^+\otimes_{\BoQ}t\BoQ[t]
\]
and
\[
 \mathbf{R}_{F^{\SN}}[W_{1+\infty}^+]\coloneqq \bigoplus_{i\geq 0}F_{2i}^{\SN}t^i\subset W_{1+\infty}^+\otimes_{\BoQ}\BoQ[t]
\]
the corresponding Rees Lie algebras.

We also recall the non-nilpotent filtration defined by Davison in \cite{davison2022affine}:
\[
 F_i\coloneqq\Span_{\BoQ}\left(z^mD^a\colon 0\leq a\leq \frac{i+2}{2}, m\geq 1\right)\,.
\]
Then, it is clear that $F_i^{\nil}=F_{i-4}$ and $F_i^{\SN}=F_{i-2}$.

The following theorem holds.

\begin{theorem}
\label{theorem:equivariant_cohas}
Let $\BoC^*$ be the one-dimensional torus rescaling the first coordinate axis of $\BoA^2$ with weight $1$ and the second coordinate axis with weight $-1$. Then, both the nilpotent and seminilpotent $\BoC^*$-equivariant CoHAs of $\BoA^2$ admit cocommutative coproducts. The Lie algebras of primitive elements $\hat{\Fg}_{\BoA^2}^{\BoC^*,\nil}$ and $\hat{\Fg}_{\BoA^2}^{\BoC^*,\SN}$ (the \emph{affinized nilpotent and seminilpotent $\BoC^*$-equivariant BPS Lie algebras}) are respectively isomorphic to $\mathbf{R}_{F^{\nil}}[W_{1+\infty}^+]$ and $\mathbf{R}_{F^{\SN}}[W_{1+\infty}^+]$. Furthermore, there are identifications of the $\BoC^*$-equivariant CoHAs
\[
 \CA_{\BoA^2}^{\BoC^*,\nil}\cong\mathbf{U}_{\BoQ[t]}(\mathbf{R}_{F^{\nil}}[W_{1+\infty}^+]),\quad\CA_{\BoA^2}^{\BoC^*,\SN}\cong\mathbf{U}_{\BoQ[t]}(\mathbf{R}_{F^{\SN}}[W_{1+\infty}^+])
\]
as $\BoQ[t]$-algebras, where $\HO^*_{\BoC^*}(\pt)\cong \BoQ[t]$.
\end{theorem}

However, when we let the bigger torus $T'\coloneqq (\BoC^*)^2$ act on $\BoA^2$ by $(t_1,t_2)\cdot (x,y)=(t_1x,t_2y)$, the nilpotent $T'$-equivariant CoHAs $\CA_{\BoA^2}^{T',\nil}$ and $\CA_{\BoA^2}^{T',\SN}$ are noncommutative and are harder to describe than the non-nilpotent CoHA $\CA_{\BoA^2}^{T'}$, which is isomorphic to the positive Yangian $\mathcal{Y}_{t_1,t_2}^+(\widehat{\mathfrak{gl}(1)})$ \cite[Theorem~5.11]{davison2022affine}. The latter identification uses the spherical generation of $\CA_{\BoA^2}^{T'}$ \cite[Theorem~B]{davison2022affine}, which fails for both $\CA_{\BoA^2}^{T',\nil}$ and $\CA_{\BoA^2}^{T',\SN}$, \cite[Proposition~5.8]{davison2022affine}. We therefore leave this identification to further investigations. Moreover, in that situation, the cohomological Hall algebra does not admit a factorization coproduct, as the procedure to construct the coproduct cannot be performed $T'$-equivariantly (the $T'$-action extends to the triple quiver and leaves the canonical cubic potential invariant, but does not leave any of the loops invariant and so, the factorization coproduct cannot be defined). Therefore, the $T'$-equivariant CoHAs are not enveloping algebras in a natural way and it is necessary to study them directly.

The proofs of Theorems~\ref{theorem:commutativity-fully-nilpotent}, \ref{theorem:commutativity-semi-nilpotent} and \ref{theorem:equivariant_cohas} are given in \Cref{section:commutativity_proof}.

\subsection*{Conventions and Notations}
We work with constructible sheaves of $\BoQ$-vector spaces and with cohomology with rational coefficients.

\subsection*{Acknowledgements} I would like to thank Sasha Minets for useful comments on a former version of this paper.

\section{The CoHA of $\BoA^2$, the BPS Lie algebra and its affinization}
\label{section:CoHA-A2}
The CoHA of $\BoA^2$ is an associative algebra structure on the Borel--Moore homology of the commuting stack $\CC(\mathfrak{gl})\coloneqq \bigsqcup_{d\geq 0}\CC(\mathfrak{gl}_d)$, where
\[
 \CC(\mathfrak{gl}_d)\coloneqq \{(x,y)\in\mathfrak{gl}_d\mid [x,y]=0\}/\GL_d
\]
denotes the stack of pairs of commuting $d\times d$-matrices. We let $\JH_d\colon \CC(\mathfrak{gl}_d)\rightarrow C(\mathfrak{gl}_d)$ be the good moduli space (or affinization) morphism where $C(\mathfrak{gl}_d)\coloneqq \{(x,y)\in\mathfrak{gl}_d\mid [x,y]=0\}\cms\GL_d\cong (\BoA^2)^d\cms\mathfrak{S}_d$ is the affine GIT quotient and the isomorphism is the Chevalley restriction isomorphism for commuting pairs (the identification of the reduced subschemes is a combination of \cite[Theorem~2.9]{joseph1997harish} and \cite[Theorem~A]{richardson1979commuting} and the reducedness of the GIT quotient is proven in \cite{gan2006almost}). We define $C(\mathfrak{gl})\coloneqq\bigsqcup_{d\in\BoN}C(\mathfrak{gl}_d)$ and $\JH\colon \CC(\mathfrak{gl})\rightarrow C(\mathfrak{gl})$.

We consider the direct sum morphism $\oplus\colon C(\mathfrak{gl})\times C(\mathfrak{gl})\rightarrow C(\mathfrak{gl})$ obtained by combining the morphisms $C(\mathfrak{gl}_d)\times C(\mathfrak{gl}_{d'})\rightarrow C(\mathfrak{gl}_{d+d'})$. This is a finite morphism (and therefore, it is perverse t-exact), and we obtain symmetric monoidal structures on $\Perv(C(\mathfrak{gl}))$ and $\CD^+_{\rmc}(C(\mathfrak{gl}),\BoQ)$ by the formula $\CF\boxdot\CG\coloneqq \oplus_*(\CF\boxtimes\CG)$ for $\CF,\CG\in\Perv(C(\mathfrak{gl}))$ (resp. $\CF,\CG\in\CD^+_{\rmc}(C(\mathfrak{gl}),\BoQ))$.

The sheafified CoHA is an algebra structure on the complex of constructible sheaves $\SA_{\BoA^2}\coloneqq\bigoplus_{d\in\BoN}\JH_{d,*}\BD\BoQ_{\CC(\mathfrak{gl}_d)}\in\CD^+_{\rmc}(C(\mathfrak{gl}),\BoQ)$ with respect to the monoidal structure $\boxdot$. The absolute CoHA $\CA_{\BoA^2}\coloneqq \HO^*(\SA_{\BoA^2})=\bigoplus_{d\geq 0}\HO^{\rmBM}_*(\CC(\mathfrak{gl}_d))$ is obtained by taking the derived global sections. The construction of the CoHA product at the sheaf level is explained in \cite[\S3.2.4]{davison2020bps} in terms of dimensional reduction from the corresponding Jacobi algebra of the triple quiver with the canonical cubic potential, and in \cite[Appendix A]{davison2022bps} directly at the level of preprojective algebras. At the level of derived sections, the construction of the product is explained in \cite{schiffmann2013cherednik} in the particular case of the Jordan quiver we are interested in, then extended to all preprojective algebras of quivers in \cite{yang2018cohomological}. The constructions of the products via dimensional reduction and at the preprojective algebra levels are compared in \cite[Appendix by Ben Davison]{ren2017cohomological} and \cite{yang2020two}: these two products coincide up to an explicit sign that is not directly relevant for our study.

We let $\Delta_d\colon\BoA^2\rightarrow C(\mathfrak{gl}_d)$, $(x,y)\mapsto ((x,y),\hdots,(x,y))$ be the small diagonal in $(\BoA^2)^d\cms \mathfrak{S}_d$.

The BPS sheaf of $\BoA^2$ (equivalently of the commuting stack) is the perverse sheaf
\[
 \BPS_{\BoA^2}=\bigoplus_{d\geq 1}\Delta_{d,*}\BoQ_{\BoA^2}[2]\in\Perv(C(\mathfrak{gl})).
\]
It is described in \cite[Theorem~5.1]{davison2016integrality}. It has the trivial Lie bracket, see \cite[Proposition~7.3]{davison2020bps}. The cohomological integrality isomorphism \cite[Theorem~D]{davison2016integrality} is an isomorphism of complexes of constructible sheaves
\[
 \JH_{*}\BD\BoQ_{\CC(\mathfrak{gl})}\cong\Sym(\BPS_{\BoA^2}\otimes\HO^*_{\BoC^*}(\pt))\in\CD^+_{\rmc}(C(\mathfrak{gl}),\BoQ).
\]
We refer to \cite{davison2020cohomological} for details regarding the symmetric power in the category of constructible complexes. Note that the virtual shift that usually appears on the constant sheaf is trivial since the virtual dimension of $\CC(\mathfrak{gl})$ is $0$. We let $\Fg_{\BoA^2}\coloneqq \HO^*(\BPS_{\BoA^2})$ be the \emph{BPS Lie algebra} of $\BoA^2$. It is commutative \cite[Proposition~7.3]{davison2020bps}. There is a cocommutative coproduct on $\CA_{\BoA^2}$ making it a bialgebra \cite[\S6.2]{davison2026hecke}, see also \cite{jindal2024coha,hennecart2026degenerations}. We denote by $\hat{\Fg}_{\BoA^2}\subset \CA_{\BoA^2}$ the subspace of primitive elements. It is a Lie algebra and by the graded Milnor--Moore theorem, there is an isomorphism of algebras $\CA_{\BoA^2}\cong\mathbf{U}(\hat{\Fg}_{\BoA^2})$. Moreover, by the support lemma for the BPS sheaf, there is an identification of vector spaces $\hat{\Fg}_{\BoA^2}\cong\HO^*(\BPS_{\BoA^2}\otimes\HO^*_{\BoC^*}(\pt))$ \cite[Proposition~3.8]{davison2022affine}. By semisimplicity of the underlying constructible complex of $\SA_{\BoA^2}$ \cite[Theorem~A]{davison2020bps}, we choose a direct sum complement $E$ of $\BPS_{\BoA^2}\otimes\HO^*_{\BoC^*}(\pt)$, so that as constructible complexes $\SA_{\BoA^2}\cong (\BPS_{\BoA^2}\otimes\HO^*_{\BoC^*}(\pt))\oplus E$. Even though we do not claim that $\BPS_{\BoA^2}\otimes\HO^*_{\BoC^*}(\pt)$ is stable under the sheafified Lie bracket of $\SA_{\BoA^2}$ (although we believe it is the case), we know by the existence of the coproduct on $\CA_{\BoA^2}$ that after taking derived global sections, it is stable under the Lie bracket. Therefore, we will abusively call Lie bracket on $\BPS_{\BoA^2}\otimes\HO^*_{\BoC^*}(\pt)$ the composition of the Lie bracket $(\BPS_{\BoA^2}\otimes\HO^*_{\BoC^*}(\pt))\boxdot(\BPS_{\BoA^2}\otimes\HO^*_{\BoC^*}(\pt))\rightarrow\SA_{\BoA^2}$ with the projection $\SA_{\BoA^2}\rightarrow\BPS_{\BoA^2}\otimes\HO^*_{\BoC^*}(\pt)$ parallel to $E$. Note that after taking derived global sections, there is no difference between the actual Lie bracket and the Lie bracket followed by projection since $\HO^*(\BPS_{\BoA^2})\otimes\HO^*_{\BoC^*}(\pt)$ is stable under the Lie bracket \cite{davison2022affine}. Therefore, although the composition of the Lie bracket and the projection may depend on the choice of $E$ at the sheaf level, it does not after taking derived global sections.

One can add a torus action to the picture. Namely, we let $T'$ be a torus acting on $\BoA^2$. Then, the sheafified CoHA $\SA_{\BoA^2}$ is $T'$-equivariant, and so defines an object $\SA_{\BoA^2}^{T'}$ in the $T'$-equivariant derived category $\CD^+_{\rmc,T'}(C(\mathfrak{gl}))$. We denote by $\CA_{\BoA^2}^{T'}$ the $\HO^*_{T'}(\pt)$-algebra obtained by taking derived global sections. The monoidal structure on $\CD^+_{\rmc,T'}(C(\mathfrak{gl}))$ is defined by the formula $\CF\boxdot \CG$ as before, noticing that if $\CF, \CG$ are $T'$-equivariant, so is $\CF\boxdot \CG$. Then, there is a $T'$-equivariant version of the BPS Lie algebra $\Fg_{\BoA^2}^{T'}$ \cite[Theorem~3.4]{davison2022affine} and it satisfies the cohomological integrality isomorphism
\[
\CA_{\BoA^2}^{T'}\cong\Sym_{\HO^*_{T'}(\pt)}(\Fg_{\BoA^2}^{T'}\otimes\HO^*_{\BoC^*}(\pt)).
\]

It is also natural to consider the action of $T'\coloneqq (\BoC^*)^2$ on $\BoA^2$ with weight $(1,0)$ on the horizontal axis and weight $(0,1)$ on the vertical coordinate axis. In this case, the CoHA $\CA_{\BoA^2}^{T'}$ does not admit a coproduct, as one cannot perform the construction of the factorization coproduct.

A standard action is also that of $\BoC^*$ on $\BoA^2$ with weight $1$ on the horizontal axis and $-1$ on the vertical axis. Then, the corresponding CoHA admits a coproduct, since the action extends to the triple quiver with canonical cubic potential with trivial action on the extra loop.

\section{Fully nilpotent and seminilpotent restrictions of the CoHA of $\BoA^2$}
\label{section:nilpotent-cohas}

In this paper, we are interested in the nilpotent variants of the cohomological Hall algebra of $\BoA^2$ recalled in \Cref{section:CoHA-A2}. They come in two different flavours: the \emph{fully nilpotent CoHA} (\Cref{subsection:fully_nilpotent_coha}) and the \emph{seminilpotent CoHA} (\Cref{subsection:seminilpotent-coha}). One sometimes considers the \emph{strictly seminilpotent CoHA} as in \cite{davison2022bps}, but for one-vertex quivers it coincides with the seminilpotent CoHA.

\subsection{The fully nilpotent CoHA}
\label{subsection:fully_nilpotent_coha}
The fully nilpotent CoHA is an algebra structure on the Borel--Moore homology of the nilpotent commuting stack $\CA_{\BoA^2}^{\nil}\coloneqq \bigoplus_{d\in\BoN}\HO^{\rmBM}_*(\CC^{\nil}(\mathfrak{gl}_d),\BoQ)$ where $\CC^{\nil}(\mathfrak{gl})$ is the commuting stack of pairs of nilpotent commuting $d\times d$ matrices, defined by the Cartesian square
\[\begin{tikzcd}
	{\CC^{\nil}(\mathfrak{gl}_d)} & {\CC(\mathfrak{gl}_d)} \\
	{\{0_d\}} & {C(\mathfrak{gl}_d)}
	\arrow[from=1-1, to=1-2]
	\arrow[from=1-1, to=2-1]
	\arrow["{\JH_d}", from=1-2, to=2-2]
	\arrow["{\imath_{\nil}}"', from=2-1, to=2-2]
    \arrow["\lrcorner"{anchor=center, pos=0.125}, draw=none, from=1-1, to=2-2]
\end{tikzcd}\]
where $0_d\in C(\mathfrak{gl}_d)$ denotes the pair of zero $d\times d$ matrices. We still denote by $\imath_{\nil}$ the morphism $\BoN\rightarrow C(\mathfrak{gl})$, $d\mapsto 0_{d}$. By base-change, there is an identification $\CA_{\BoA^2}^{\nil}\cong\imath_{\nil}^!\SA_{\BoA^2}$ and the product on $\imath^!_{\nil}\SA_{\BoA^2}$ is obtained by applying $\imath_{\nil}^!$ to the product on $\SA_{\BoA^2}$. Here, we identify the complex of constructible sheaves $\imath^!_{\nil}\SA_{\BoA^2}$ on $\BoN$ with a $\BoN$-graded vector space. One can define a factorization coproduct on $\CA_{\BoA^2}^{\nil}$ exactly as in \cite[\S6.2]{davison2026hecke}. These structures make $\CA_{\BoA^2}^{\nil}$ a connected cocommutative bialgebra. We define $\Fg_{\BoA^2}^{\nil}\coloneqq \imath_{\nil}^!\BPS_{\BoA^2}$ and let $\hat{\Fg}_{\BoA^2}^{\nil}\subset \CA_{\BoA^2}^{\nil}$ be the Lie algebra of primitive elements. Then, $\hat{\Fg}_{\BoA^2}^{\nil}$ coincides with the subspace $\imath^!_{\nil}(\BPS\otimes\HO^*_{\BoC^*}(\pt))\subset\CA_{\BoA^2}^{\nil}$ (by the support lemma for the BPS sheaf and the cohomological integrality isomorphism for $\CA_{\BoA^2}^{\nil}$, exactly as in \cite[Proposition~3.8]{davison2022affine}) and there is an isomorphism of algebras $\CA_{\BoA^2}^{\nil}\cong\mathbf{U}(\hat{\Fg}_{\BoA^2}^{\nil})$ by the graded Milnor--Moore theorem.

We have $\Fg_{\BoA^2}^{\nil}\cong \bigoplus_{d\geq 1}\BoQ_{\{0_d\}}[-2]$ since $(\{0\}\rightarrow\BoA^2)^!\BoQ_{\BoA^2}[2]\cong \BoQ_{\{0\}}[-2]$. It is entirely concentrated in cohomological degree $2$.

\subsection{The seminilpotent CoHA}
\label{subsection:seminilpotent-coha}
The seminilpotent CoHA is an algebra structure on the Borel--Moore homology of the seminilpotent commuting stack $\CA_{\BoA^2}^{\SN}\coloneqq \bigoplus_{d\in\BoN}\HO^{\rmBM}_*(\CC^{\SN}(\mathfrak{gl}_d),\BoQ)$ where $\CC^{\SN}(\mathfrak{gl}_d)$ is the stack of commuting pairs of $d\times d$ matrices with the second matrix nilpotent. For $d\geq 0$, we let $\CM^{\SN}_d\coloneqq\mathfrak{gl}_d\cms\GL_d\rightarrow C(\mathfrak{gl}_d), x\mapsto (x,0)$ be the closed embedding. Then, the seminilpotent stack is defined by the Cartesian diagram
\[\begin{tikzcd}
	{\CC^{\SN}(\mathfrak{gl}_d)} & {\CC(\mathfrak{gl}_d)} \\
	{\CM^{\SN}_d} & {C(\mathfrak{gl}_d)}
	\arrow[from=1-1, to=1-2]
	\arrow[from=1-1, to=2-1]
	\arrow["\lrcorner"{anchor=center, pos=0.125}, draw=none, from=1-1, to=2-2]
	\arrow["{\JH_d}", from=1-2, to=2-2]
	\arrow["{\imath_{\SN}}"', from=2-1, to=2-2]
\end{tikzcd}\]
We still denote by $\imath_{\SN}\colon\bigsqcup_{d\in\BoN}\CM_d^{\SN}\rightarrow C(\mathfrak{gl})$. The sheafified seminilpotent CoHA is $\imath_{\SN}^!\SA_{\BoA^2}$ with product obtained by application of $\imath_{\SN}^!$. By base-change, there is an identification $\CA_{\BoA^2}^{\SN}\cong \HO^*(\imath_{\SN}^!\SA_{\BoA^2})$. One can construct a factorization coproduct on $\CA_{\BoA^2}^{\SN}$ exactly as in \cite[\S6.2]{davison2026hecke}. The product and coproduct make $\CA_{\BoA^2}^{\SN}$ a connected cocommutative bialgebra. We define $\BPS_{\BoA^2}^{\SN}\coloneqq\imath_{\SN}^!\BPS_{\BoA^2}$, $\Fg^{\SN}_{\BoA^2}\coloneqq \HO^*(\BPS_{\BoA^2}^{\SN})$, and we let $\hat{\Fg}_{\BoA^2}^{\SN}\subset\CA_{\BoA^2}^{\SN}$ be the Lie algebra of primitive elements. Then, $\hat{\Fg}_{\BoA^2}^{\SN}$ coincides with the subspace $\HO^*(\BPS_{\BoA^2}^{\SN}\otimes\HO^*_{\BoC^*}(\pt))\subset\CA_{\BoA^2}^{\SN}$ (by the support theorem for the BPS sheaf and the cohomological integrality isomorphism as in \cite[Proposition~3.8]{davison2022affine}). By the graded Milnor--Moore theorem, there is an isomorphism of algebras $\CA_{\BoA^2}^{\SN}\cong\mathbf{U}(\hat{\Fg}_{\BoA^2}^{\SN})$.

We let $\delta_d\colon \BoA^1\rightarrow C(\mathfrak{gl}_d)$, $x\mapsto (xI_d,0)$. We have $\BPS_{\BoA^2}^{\SN}\cong\bigoplus_{d\geq 1}\delta_{d,*}\BoQ_{\BoA^1}$ since $(\BoA^1\rightarrow\BoA^2)^!\BoQ_{\BoA^2}[2]\cong\BoQ_{\BoA^1}$ and $\Fg_{\BoA^2}^{\SN}\cong \bigoplus_{d\geq 1}\BoQ$ by taking derived global sections, since $\HO^*(\BoA^1)=\BoQ$. It is entirely concentrated in cohomological degree $0$.

\subsection{Equivariant versions}
\label{subsection:equivariant_versions}
Since the nilpotent and seminilpotent loci are invariant under the actions of $T'=(\BoC^*)^2$ and of $\BoC^*$ defined in the last two paragraphs of \Cref{section:CoHA-A2} (respectively with weights $(1,0),(0,1)$ and $1,-1$ on the coordinate axes of $\BoA^2$), one can define the $T'$-equivariant and $\BoC^*$-equivariant nilpotent CoHAs $\CA_{\BoA^2}^{T',\nil}$, $\CA_{\BoA^2}^{T',\SN}$, $\CA_{\BoA^2}^{\BoC^*,\nil}$ and $\CA_{\BoA^2}^{\BoC^*,\SN}$ as in \Cref{subsection:fully_nilpotent_coha} and \Cref{subsection:seminilpotent-coha}.

\section{Commutativity of the affinized nilpotent BPS Lie algebras and equivariant versions}

\label{section:commutativity_proof}

\begin{proof}[Proof of \Cref{theorem:commutativity-fully-nilpotent}]
We show that the Lie bracket on the affinized fully nilpotent BPS Lie algebra $\hat{\Fg}^{\nil}_{\BoA^2}$ vanishes.

The Lie bracket on $\BPS_{\BoA^2}\otimes\HO^*_{\BoC^*}(\pt)$ respects the perverse filtration. This comes from the fact that the CoHA product on $\SA_{\BoA^2}$ is given by a morphism of constructible complexes $\SA_{\BoA^2}\boxdot\SA_{\BoA^2}\rightarrow\SA_{\BoA^2}$ and so, using standard properties of $t$-structures, it respects the perverse filtration, and the same becomes immediately true for the Lie bracket. This means that for any $m,n\in\BoN$, identifying $\BPS_{\BoA^2}\otimes\HO^*_{\BoC^*}(\pt)\cong\bigoplus_{m\geq 0}\BPS_{\BoA^2}[-2m]$ as constructible complexes, the Lie bracket sends $\BPS_{\BoA^2}[-2m]\boxdot\BPS_{\BoA^2}[-2n]$ to $\bigoplus_{k\leq m+n}\BPS_{\BoA^2}[-2k]$. Applying $\imath^!_{\nil}$, we obtain that the Lie bracket on $\hat{\Fg}^{\nil}_{\BoA^2}$ sends $\Fg^{\nil}_{\BoA^2}u^m\otimes\Fg^{\nil}_{\BoA^2}u^n$ to $\Fg^{\nil}_{\BoA^2,\leq m+n}\coloneqq\bigoplus_{k\leq m+n}\Fg^{\nil}_{\BoA^2}u^k$ where $u\in\HO^*_{\BoC^*}(\pt)$ is a degree $2$ generator. The Lie bracket is also graded for the cohomological degree. Since $\Fg^{\nil}_{\BoA^2}u^m\otimes\Fg^{\nil}_{\BoA^2}u^n$ sits in cohomological degree $2+2m+2+2n=4+2(m+n)$ and $\Fg^{\nil}_{\BoA^2,\leq m+n}$ sits in cohomological degrees $\leq 2+2(m+n)$, the Lie bracket on $\Fg^{\nil}_{\BoA^2}u^m\otimes\Fg^{\nil}_{\BoA^2}u^n$ vanishes. This proves that $\hat{\Fg}^{\nil}_{\BoA^2}$ is an abelian Lie algebra, and therefore that $\CA_{\BoA^2}^{\nil}\cong\mathbf{U}(\hat{\Fg}_{\BoA^2}^{\nil})$ is a commutative algebra.
\end{proof}

\begin{proof}[Proof of \Cref{theorem:commutativity-semi-nilpotent}]
As explained in \Cref{subsection:seminilpotent-coha}, we have
\[
 \imath_{\SN}^!\BPS_{\BoA^2}\cong \bigoplus_{d\geq 1}\delta_{d,*}\BoQ_{\BoA^1}.
\]
Moreover, since the associated graded with respect to the less perverse filtration of the Lie bracket on $\BPS_{\BoA^2}\otimes\HO^*_{\BoC^*}(\pt)$ vanishes (see \cite[Theorem~1.1]{hennecart2026degenerations}) by the fact that the Lie bracket on $\BPS_{\BoA^2}$ vanishes \cite[Proposition~7.3]{davison2020bps}, the Lie bracket on $\BPS_{\BoA^2}\otimes\HO^*_{\BoC^*}(\pt)$ sends $\BPS_{\BoA^2}[-2m]\boxdot\BPS_{\BoA^2}[-2n]$ to $\BPS_{\BoA^2}[-2k]$ for $k<m+n$. By applying $\imath_{\SN}^!$ and taking derived global sections, the Lie bracket on $\hat{\Fg}^{\SN}_{\BoA^2}$ sends $\Fg^{\SN}_{\BoA^2}u^m\otimes\Fg^{\SN}_{\BoA^2}u^n$ to $\Fg^{\SN}_{\BoA^2,<m+n}\coloneqq\bigoplus_{k<m+n}\Fg^{\SN}_{\BoA^2}u^k$. Since $\Fg^{\SN}_{\BoA^2}$ lies entirely in cohomological degree $0$, then $\Fg^{\SN}_{\BoA^2}u^m\otimes\Fg^{\SN}_{\BoA^2}u^n$ lies entirely in cohomological degree $2(m+n)$ and $\Fg^{\SN}_{\BoA^2,<m+n}$ lies in cohomological degrees $<2(m+n)$. Therefore, since the Lie bracket is graded with respect to the cohomological degree, it must vanish. This proves the commutativity of $\hat{\Fg}^{\SN}_{\BoA^2}$ and therefore of $\CA_{\BoA^2}^{\SN}\cong\mathbf{U}(\hat{\Fg}^{\SN}_{\BoA^2})$.
\end{proof}

Before turning to the proof of \Cref{theorem:equivariant_cohas} regarding the equivariant versions of nilpotent CoHAs of $\BoA^2$, we give a couple of lemmas that will be instrumental.

\begin{lemma}
\label{lemma:identifications}
Let $\BoC^*$ act on the commuting stack $\CC(\mathfrak{gl})$ with weight $1$ on the first coordinate and weight $-1$ on the second coordinate. This action descends to the GIT quotient $C(\mathfrak{gl})$. Then, for any $d,e\in\BoN_{\geq 1}$, $m,n\geq 0$, we have
\begin{multline}
 \Hom_{\CD_{\rmc,\BoC^*}^{\rmb}(C(\mathfrak{gl}_{d+e}))}(\Delta_{d,*}\BoQ_{\BoA^2}[2-2m]\boxdot\Delta_{e,*}\BoQ_{\BoA^2}[2-2n],\Delta_{d+e,*}\BoQ_{\BoA^2}[2-2k])\cong\\ \HO^{2(m+n-k-1)}_{\BoC^*}(\pt)\cong
 \left\{\begin{aligned}
 &\BoQ\cdot\hbar^{m+n-k-1} \quad&\text{if $k\leq m+n-1$}\\
 &0\quad&\text{otherwise}
 \end{aligned}\right.
\end{multline}
\begin{multline}
 \Hom_{\CD_{\rmc,\BoC^*}^{\rmb}(\CM_{d+e}^{\SN})}(\delta_{d,*}\BoQ_{\BoA^1}[-2m]\boxdot\delta_{e,*}\BoQ_{\BoA^1}[-2n],\delta_{d+e,*}\BoQ_{\BoA^1}[-2k])\cong\\ \HO^{2(m+n-k)}_{\BoC^*}(\pt)\cong
 \left\{\begin{aligned}
 &\BoQ\cdot\hbar^{m+n-k} \quad&\text{if $k\leq m+n$}\\
 &0\quad&\text{otherwise}
 \end{aligned}\right.
\end{multline}
and
\begin{multline}
\Hom_{\CD^{\rmb}_{\rmc,\BoC^*}(\pt)}
\left(
\BoQ_{\{0_d\}}[-2-2m]\otimes
\BoQ_{\{0_e\}}[-2-2n],
\BoQ_{\{0_{d+e}\}}[-2-2k]
\right)
\\
\cong
\HO^{2(m+n-k+1)}_{\BoC^*}(\pt)
\cong
\begin{cases}
\BoQ\cdot \hbar^{m+n-k+1}, & \text{if } k\leq m+n+1,\\
0, & \text{otherwise.}
\end{cases}
\end{multline}
identifying $\HO^*_{\BoC^*}(\pt)\cong\BoQ[\hbar]$.
\end{lemma}
\begin{proof}
The three isomorphisms can be proven in the exact same way. We only prove the first one. Let
\[
g_{d,e}\colon \BoA^2\times \BoA^2\longrightarrow C(\mathfrak{gl}_{d+e})
\]
be the morphism sending a pair of points $((p,q),(r,s))$ to the pair of commuting matrices $(\diag(pI_d,rI_e),\diag(qI_d,sI_e))$. This is a finite morphism since it is the composition of the closed immersion $\Delta_d\boxtimes\Delta_e \colon\BoA^2\times\BoA^2\rightarrow C(\mathfrak{gl}_d)\times C(\mathfrak{gl}_e)$ with the direct sum morphism $\oplus\colon C(\mathfrak{gl}_d)\times C(\mathfrak{gl}_e)\rightarrow C(\mathfrak{gl}_{d+e})$, which is finite. Therefore, $g_{d,e}$ is proper and so the functors $g_{d,e,*}$ and $g_{d,e,!}$ coincide. Then
\[
\Delta_{d,*}\BoQ_{\BoA^2}[2-2m]\boxdot
\Delta_{e,*}\BoQ_{\BoA^2}[2-2n]
\simeq
g_{d,e,*}\BoQ_{\BoA^2\times\BoA^2}[4-2m-2n].
\]
The inverse image of the small diagonal
\[
\Delta_{d+e}(\BoA^2)\subset C(\mathfrak{gl}_{d+e})
\]
under $g_{d,e}$ is the diagonal copy of $\BoA^2$ in
$\BoA^2\times\BoA^2$. We denote by
\[
j\colon \BoA^2\hookrightarrow \BoA^2\times\BoA^2
\]
this diagonal embedding. Thus we have a Cartesian diagram
\[\begin{tikzcd}
	{\BoA^2} & {\BoA^2\times\BoA^2} \\
	{\BoA^2} & {C(\mathfrak{gl}_{d+e})}
	\arrow["j", from=1-1, to=1-2]
	\arrow["\id"', from=1-1, to=2-1]
	\arrow["\lrcorner"{anchor=center, pos=0.125}, draw=none, from=1-1, to=2-2]
	\arrow["{g_{d,e}}", from=1-2, to=2-2]
	\arrow["{\Delta_{d+e}}"', from=2-1, to=2-2]
\end{tikzcd}\]
Set $a=4-2m-2n$ and $b=2-2k$.

By adjunction $g_{d,e,*}\cong g_{d,e,!}\dashv g_{d,e}^!$, base change for the above Cartesian square, and the adjunction $j^*\dashv j_*$, we obtain
\begin{align*}
&\Hom_{\CD_{\rmc,\BoC^*}^{\rmb}(C(\mathfrak{gl}_{d+e}))}
\left(
g_{d,e,*}\BoQ_{\BoA^2\times\BoA^2}[a],
\Delta_{d+e,*}\BoQ_{\BoA^2}[b]
\right) \\
&\cong
\Hom_{\CD_{\rmc,\BoC^*}^{\rmb}(\BoA^2\times\BoA^2)}
\left(
\BoQ_{\BoA^2\times\BoA^2}[a],
g_{d,e}^!\Delta_{d+e,*}\BoQ_{\BoA^2}[b]
\right) \\
&\cong
\Hom_{\CD_{\rmc,\BoC^*}^{\rmb}(\BoA^2\times\BoA^2)}
\left(
\BoQ_{\BoA^2\times\BoA^2}[a],
j_*\BoQ_{\BoA^2}[b]
\right) \\
&\cong
\Hom_{\CD_{\rmc,\BoC^*}^{\rmb}(\BoA^2)}
\left(
j^*\BoQ_{\BoA^2\times\BoA^2}[a],
\BoQ_{\BoA^2}[b]
\right) \\
&\cong
\Hom_{\CD_{\rmc,\BoC^*}^{\rmb}(\BoA^2)}
\left(
\BoQ_{\BoA^2}[a],
\BoQ_{\BoA^2}[b]
\right).
\end{align*}
In the first and second isomorphisms, $\BoC^*$ acts diagonally on $\BoA^2\times\BoA^2$. Therefore the required Hom group is
\[
\HO^{b-a}_{\BoC^*}(\BoA^2)
=
\HO^{2(m+n-k-1)}_{\BoC^*}(\BoA^2).
\]
Since $\BoA^2$ is $\BoC^*$-equivariantly contractible and
\[
\HO^*_{\BoC^*}(\pt)\cong \BoQ[\hbar],
\qquad
\deg(\hbar)=2,
\]
we have
\[
\HO^{2(m+n-k-1)}_{\BoC^*}(\BoA^2)
\cong
\begin{cases}
\BoQ\cdot\hbar^{m+n-k-1}, & \text{if } k\leq m+n-1,\\
0, & \text{otherwise.}
\end{cases}
\]
This proves the first identification.
\end{proof}

\begin{corollary}
The $\BoC^*$-equivariant sheafified Lie bracket on $\BPS_{\BoA^2}\otimes\HO^*_{\BoC^*}$ is determined by the $\BoC^*$-equivariant absolute Lie bracket on $\hat{\Fg}_{\BoA^2}$.
\end{corollary}
\begin{proof}
By \Cref{lemma:identifications}, every component of the sheafified Lie bracket
\[
 \left(\BPS_{\BoA^2}\otimes \HO^*_{\BoC^*}(\pt)
 \right)\boxdot
 \left(\BPS_{\BoA^2}\otimes \HO^*_{\BoC^*}(\pt)\right)
 \longrightarrow
 \BPS_{\BoA^2}\otimes \HO^*_{\BoC^*}(\pt)
\]
is a morphism of the form
\[
 \Delta_{d,*}\BoQ_{\BoA^2}[2-2m]\boxdot
 \Delta_{e,*}\BoQ_{\BoA^2}[2-2n]
 \longrightarrow
 \Delta_{d+e,*}\BoQ_{\BoA^2}[2-2k]
\]
for some $m,n,k\in\BoN$. By \Cref{lemma:identifications}, the corresponding equivariant Hom space is either zero or one-dimensional
over $\BoQ$. More precisely, when $k\leq m+n-1$, it is generated by
\[
 \hbar^{m+n-k-1},
\]
and it is zero otherwise. Hence each component of the sheafified bracket is
determined by a single scalar.

Taking $\BoC^*$-equivariant global sections sends this generator to
multiplication by the same class
\[
 \hbar^{m+n-k-1}
\]
on equivariant cohomology, that is $\HO^{2-2m+2-2n}_{\BoC^*}(\BoA^2\times\BoA^2)\cong\BoQ[\hbar]\rightarrow\HO^{2-2k}_{\BoC^*}(\BoA^2)\cong\BoQ[\hbar]$. Since
\[
 \HO^*_{\BoC^*}(\pt)\cong \BoQ[\hbar]
\]
is an integral domain, multiplication by a nonzero power of $\hbar$ is
nonzero. Therefore a component of the sheafified bracket vanishes if and
only if the induced map on equivariant global sections vanishes. Moreover,
the scalar coefficient of each component is recovered from the induced bracket on global sections.

Thus the $\BoC^*$-equivariant sheafified Lie bracket is completely
determined by the induced $\BoC^*$-equivariant Lie bracket on equivariant
global sections, namely on the affinized BPS Lie algebra
\[
 \hat{\Fg}_{\BoA^2}^{\BoC^*}
 =
 \HO^*_{\BoC^*}
 \left(
 \BPS_{\BoA^2}\otimes \HO^*_{\BoC^*}(\pt)
 \right).
\]
\end{proof}

Now that we have identified the $\BoC^*$-equivariant sheafified Lie bracket on $\BPS_{\BoA^2}\otimes\HO^*(\pt)$, we only need to understand the effect of the restrictions $\imath_{\nil}^!$ and $\imath_{\SN}^!$ on the components of the Lie bracket.

\begin{lemma}
\label{lemma:restriction-euler-classes}
 Via the identifications in \Cref{lemma:identifications}, for any $m,n,k\in\BoN$, the morphism
 \begin{multline}
 \label{first}
  \imath_{\SN}^!\colon\Hom_{\CD_{\rmc,\BoC^*}^{\rmb}(C(\mathfrak{gl}_{d+e}))}(\Delta_{d,*}\BoQ_{\BoA^2}[2-2m]\boxdot\Delta_{e,*}\BoQ_{\BoA^2}[2-2n],\Delta_{d+e,*}\BoQ_{\BoA^2}[2-2k])\rightarrow\\\Hom_{\CD_{\rmc,\BoC^*}^{\rmb}(\CM_{d+e}^{\SN})}(\delta_{d,*}\BoQ_{\BoA^1}[-2m]\boxdot\delta_{e,*}\BoQ_{\BoA^1}[-2n],\delta_{d+e,*}\BoQ_{\BoA^1}[-2k])
 \end{multline}
is given by the multiplication by $-\hbar$ (the $\BoC^*$-equivariant Euler class of the normal bundle of $\BoA^1\subset\BoA^2$) and the morphism
\begin{multline}
  \imath_{\nil}^!\colon\Hom_{\CD_{\rmc,\BoC^*}^{\rmb}(C(\mathfrak{gl}_{d+e}))}(\Delta_{d,*}\BoQ_{\BoA^2}[2-2m]\boxdot\Delta_{e,*}\BoQ_{\BoA^2}[2-2n],\Delta_{d+e,*}\BoQ_{\BoA^2}[2-2k])\rightarrow\\
  \Hom_{\CD^{\rmb}_{\rmc,\BoC^*}}
\left(
\BoQ_{\{0_d\}}[-2-2m]\otimes
\BoQ_{\{0_e\}}[-2-2n],
\BoQ_{\{0_{d+e}\}}[-2-2k]
\right)
 \end{multline}
is given by the multiplication by $-\hbar^2$ (the $\BoC^*$-equivariant Euler class of the normal bundle of $\{0\}\subset \BoA^2$).
\end{lemma}
\begin{proof}
By \Cref{lemma:identifications}, the source Hom space in both morphisms is
generated by $\hbar^{m+n-k-1}$.

We first consider the seminilpotent restriction
\[
 \imath_{\SN}\colon \BoA^1\hookrightarrow \BoA^2,
\]
where $\BoA^1$ is the horizontal axis. Its normal bundle is the normal line
given by the second coordinate. Since the $\BoC^*$-action has weight $-1$ on
the second coordinate, this normal line is the one-dimensional representation
$\BoC_{-1}$. Hence
\[
 e_{\BoC^*}(\BoC_{-1})=-\hbar.
\]
The functor $\imath_{\SN}^!$ is the equivariant Gysin restriction. Under the
identifications of \Cref{lemma:identifications} (see also the proof), this Gysin restriction in \Cref{first} is identified with
\begin{equation}
\label{equation_sn_restriction}
 \Hom_{\CD^{\rmb}_{\rmc,\BoC^*}(\BoA^2\times\BoA^2)}(\BoQ_{\BoA^2\times\BoA^2}[a],j_*\BoQ_{\BoA^2}[b])\rightarrow \Hom_{\CD^{\rmb}_{\rmc,\BoC^*}(\BoA^1\times\BoA^1)}(\BoQ_{\BoA^1\times\BoA^1}[-4+a],j'_*\BoQ_{\BoA^1}[-2+b])
\end{equation}
where $j\colon\BoA^2\rightarrow\BoA^2\times\BoA^2$ and $j'\colon\BoA^1\rightarrow\BoA^1\times\BoA^1$ are the diagonal embeddings. We have a Cartesian square
\[\begin{tikzcd}
	{\BoA^2} & {\BoA^2\times\BoA^2} \\
	{\BoA^1} & {\BoA^1\times\BoA^1}
	\arrow["j", from=1-1, to=1-2]
	\arrow["\imath", from=2-1, to=1-1]
	\arrow["\lrcorner"{anchor=center, pos=0.125, rotate=90}, draw=none, from=2-1, to=1-2]
	\arrow["{j'}"', from=2-1, to=2-2]
	\arrow["{\imath\times\imath}"', from=2-2, to=1-2]
\end{tikzcd}\]
and the equivariant Euler class of the normal bundle of $\imath$ is $-\hbar$.
Therefore, by the equivariant excess intersection formula, the morphism in \Cref{equation_sn_restriction} is given by the multiplication by $-\hbar$. Thus $\imath_{\SN}^!$ is given by multiplication by $-\hbar$.

We now consider the fully nilpotent restriction
\[
 \imath_{\nil}\colon \{0\}\hookrightarrow \BoA^2.
\]
Its normal bundle is the tangent representation at the origin,
\[
 T_0\BoA^2\cong \BoC_1\oplus \BoC_{-1}.
\]
Therefore its equivariant Euler class is
\[
 e_{\BoC^*}(T_0\BoA^2)
 =
 e_{\BoC^*}(\BoC_1)e_{\BoC^*}(\BoC_{-1})
 =
 \hbar(-\hbar)
 =
 -\hbar^2.
\]
Hence, under the identifications of
\Cref{lemma:identifications}, the Gysin restriction sends
\[
 \hbar^{m+n-k-1}
\]
to
\[
 -\hbar^2\cdot \hbar^{m+n-k-1}
 =
 -\hbar^{m+n-k+1}.
\]
Thus $\imath_{\nil}^!$ is given by multiplication by $-\hbar^2$.
\end{proof}

\begin{proof}[Proof of \Cref{theorem:equivariant_cohas}]
The nilpotent $\BoC^*$-equivariant CoHAs $\CA_{\BoA^2}^{\BoC^*,\nil}$ and $\CA_{\BoA^2}^{\BoC^*,\SN}$ admit coproducts constructed as in the non-equivariant case. The reason is that one can extend the $\BoC^*$-action to the triple Jordan quiver (by acting trivially on the extra loop) so that the canonical cubic potential is invariant. The fact that the action on the extra loop is trivial implies that one can take open balls as in the construction of the coproduct, see \cite{jindal2024coha,hennecart2026degenerations,davison2026hecke}. Therefore, the corresponding affinized BPS Lie algebras $\hat{\Fg}_{\BoA^2}^{\BoC^*,\nil}$ and $\hat{\Fg}_{\BoA^2}^{\BoC^*,\SN}$ can be defined as the Lie algebras of primitive elements. Then, by the Milnor--Moore theorem, there are isomorphisms $\CA_{\BoA^2}^{\BoC^*,\nil}\cong\mathbf{U}_{\BoQ[t]}(\hat{\Fg}_{\BoA^2}^{\BoC^*,\nil})$ and $\CA_{\BoA^2}^{\BoC^*,\SN}\cong\mathbf{U}_{\BoQ[t]}(\hat{\Fg}_{\BoA^2}^{\BoC^*,\SN})$ and it suffices to describe the nilpotent affinized BPS Lie algebras. In this context of $\BoQ[t]$-algebras, the Milnor--Moore is explained in \Cref{section:milnor-moore-rings}, using that both $\CA_{\BoA^2}^{\BoC^*,\nil}$ and $\CA_{\BoA^2}^{\BoC^*,\SN}$ are free $\HO^*_{\BoC^*}(\pt)\cong\BoQ[t]$-modules and are generated by $P^{\nil}=\mathfrak{g}_{\BoA^2}^{\BoC^*,\nil}\otimes\HO^*_{\BoC^*}(\pt)$ and $P^{\SN}=\mathfrak{g}_{\BoA^2}^{\BoC^*,\SN}\otimes\HO^*_{\BoC^*}(\pt)$ respectively, by the respective cohomological integrality isomorphisms for $\CA_{\BoA^2}^{\BoC^*,\nil}$ and $\CA_{\BoA^2}^{\BoC^*,\SN}$ (which follow by $\imath_{\nil}^!$ or $\imath_{\SN}^!$-restrictions from the cohomological integrality isomorphism for $\SA_{\BoA^2}$ at the sheaf level).

It remains to identify the $\BoC^*$-equivariant affinized BPS Lie algebras. We write
\[
 \HO^*_{\BoC^*}(\pt)=\BoQ[\hbar],
\]
and we identify the Rees parameter $t$ with $\hbar$.

For $d\geq 1$ and $a\geq 0$, let $p_{d,a}^{\SN}$ be the generator of the seminilpotent $\BoC^*$-equivariant affinized BPS Lie algebra corresponding to the summand $\delta_{d,*}\BoQ_{\BoA^1}[-2a]$
of $\imath^!_{\SN}\BPS_{\BoA^2}$ and let $p_{d,a}^{\nil}$ be the generator of the fully nilpotent $\BoC^*$-equivariant affinized BPS Lie algebra corresponding to the summand $\BoQ_{\{0_d\}}[-2-2a]$ of $\imath^!_{\nil}\BPS_{\BoA^2}$. Similarly, let $p_{d,a}$ denote the corresponding generator of the non-nilpotent $\BoC^*$-equivariant affinized BPS Lie algebra corresponding to the summand $\BoQ_{\BoA^2}[2-2a]$ of $\BPS_{\BoA^2}$.

By Davison's description of the non-nilpotent $\BoC^*$-equivariant affinized BPS Lie algebra \cite[Theorem~C]{davison2022affine} in terms of a Rees Lie algebra associated with a filtration on $W_{1+\infty}^+$, we may choose the generators $p_{d,a}$ so that, if
\[
 [z^dD^a,z^eD^b]
 =
 \sum_c \lambda_c z^{d+e}D^c
\]
in $W^+_{1+\infty}$, then the $\BoC^*$-equivariant bracket satisfies
\[
 [p_{d,a},p_{e,b}]
 =
 \sum_c \lambda_c \hbar^{a+b-c-1}p_{d+e,c}.
\]
Here the sum is finite and $c\leq a+b-1$.

We now apply \Cref{lemma:restriction-euler-classes}. For the seminilpotent restriction, $\imath_{\SN}^!$ acts on the relevant Hom spaces by multiplication by the equivariant Euler class of the normal bundle of
\[
 \BoA^1\subset \BoA^2.
\]
Therefore
\[
 [p_{d,a}^{\SN},p_{e,b}^{\SN}]
 =
 -\sum_c \lambda_c \hbar^{a+b-c}p_{d+e,c}^{\SN}.
\]

Define a $\BoQ[t]$-linear map
\[
 \Phi_{\SN}\colon
 \hat{\Fg}^{\BoC^*,\SN}_{\BoA^2}
 \longrightarrow
 \mathbf{R}_{F^{\SN}}[W^+_{1+\infty}]
\]
by
\[
 \Phi_{\SN}\bigl(\hbar^r p_{d,a}^{\SN}\bigr)
 =
 -z^dD^a t^{a+r}.
\]
This is an isomorphism of $\BoQ[t]$-modules, since
\[
 \mathbf{R}_{F^{\SN}}[W^+_{1+\infty}]
 =
 \Span_{\BoQ}
 \left\{
 z^dD^a t^i
 \mid
 d\geq 1,\ a\geq 0,\ i\geq a
 \right\}.
\]
Moreover,
\[
 \Phi_{\SN}
 \left(
 [p_{d,a}^{\SN},p_{e,b}^{\SN}]
 \right)
 =
 \sum_c
 \lambda_c z^{d+e}D^c t^{a+b}.
\]
On the other hand,
\[
 [\Phi_{\SN}(p_{d,a}^{\SN}),\Phi_{\SN}(p_{e,b}^{\SN})]
 =
 [-z^dD^a t^a,-z^eD^b t^b].
\]
The two minus signs cancel, and since $t$ is central, this is equal to
\[
 [z^dD^a,z^eD^b]t^{a+b}
 =
 \sum_c
 \lambda_c z^{d+e}D^c t^{a+b}.
\]
Thus $\Phi_{\SN}$ is an isomorphism of Lie algebras.

For the fully nilpotent restriction, $\imath_{\nil}^!$ acts by multiplication by the equivariant Euler class of the normal bundle of
\[
 \{0\}\subset \BoA^2.
\]
which is $-\hbar^2$. Hence
\[
 [p_{d,a}^{\nil},p_{e,b}^{\nil}]
 =-\sum_c
 \lambda_c \hbar^{a+b-c+1}p_{d+e,c}^{\nil}.
\]

Define a $\BoQ[t]$-linear map
\[
 \Phi_{\nil}\colon
 \hat{\Fg}^{\BoC^*,\nil}_{\BoA^2}
 \longrightarrow
 \mathbf{R}_{F^{\nil}}[W^+_{1+\infty}]
\]
by
\[
 \Phi_{\nil}\bigl(\hbar^r p_{d,a}^{\nil}\bigr)
 =
 -z^dD^a t^{a+1+r}.
\]
This is an isomorphism of $\BoQ[t]$-modules, since
\[
 \mathbf{R}_{F^{\nil}}[W^+_{1+\infty}]
 =
 \Span_{\BoQ}
 \left\{
 z^dD^a t^i
 \mid
 d\geq 1,\ a\geq 0,\ i\geq a+1
 \right\}.
\]
Moreover,
\[
 \Phi_{\nil}
 \left(
 [p_{d,a}^{\nil},p_{e,b}^{\nil}]
 \right)
 =
 \sum_c
 \lambda_c z^{d+e}D^c t^{a+b+2}.
\]
On the other hand,
\[
 [\Phi_{\nil}(p_{d,a}^{\nil}),\Phi_{\nil}(p_{e,b}^{\nil})]
 =
 [-z^dD^a t^{a+1},-z^eD^b t^{b+1}].
\]
Since $t$ is central, this is equal to
\[
 [z^dD^a,z^eD^b]t^{a+b+2}
 =
 \sum_c
 \lambda_c z^{d+e}D^c t^{a+b+2}.
\]
Thus $\Phi_{\nil}$ is an isomorphism of Lie algebras.

We have proved
\[
 \hat{\Fg}^{\BoC^*,\nil}_{\BoA^2}
 \cong
 \mathbf{R}_{F^{\nil}}[W^+_{1+\infty}]
\]
and
\[
 \hat{\Fg}^{\BoC^*,\SN}_{\BoA^2}
 \cong
 \mathbf{R}_{F^{\SN}}[W^+_{1+\infty}].
\]
Together with the Milnor--Moore isomorphisms established above, this gives
\[
 \CA^{\BoC^*,\nil}_{\BoA^2}
 \cong
 \mathbf U_{\BoQ[t]}
 \left(
 \mathbf{R}_{F^{\nil}}[W^+_{1+\infty}]
 \right),
\]
and
\[
 \CA^{\BoC^*,\SN}_{\BoA^2}
 \cong
 \mathbf U_{\BoQ[t]}
 \left(
 \mathbf{R}_{F^{\SN}}[W^+_{1+\infty}]
 \right).
\]
This completes the proof.
\end{proof}



\appendix

\section{Milnor--Moore theorem over rings}
\label{section:milnor-moore-rings}
We recall the following useful lemma \cite[Appendix 2, p.880]{LangAlgebra}.
\begin{lemma}
\label{lemma:submodule_free_is_free}
Let $R$ be a principal ideal domain. Then, any submodule of a free $R$-module is free.
\end{lemma}

We are now ready to prove a version of the Milnor--Moore theorem for algebras over PIDs.
\begin{proposition}
 Let $R$ be a $\BoQ$-algebra that is a PID, $A$ a connected cocommutative $R$-bialgebra that is free as an $R$-module and $P\subset \Prim(A)$ a submodule of the space of primitive elements of $A$ that generates $A$ as an $R$-algebra. Then, $A\cong \mathbf{U}_{R}(\Prim(A))$.
\end{proposition}

\begin{proof}
 It is clear that $\Prim(A)$ is a Lie algebra for the Lie bracket induced by the commutator Lie bracket in $A$. Then, by universality of the enveloping algebra, there is a natural morphism $\Phi\colon\mathbf{U}_R(\Prim(A))\rightarrow A$. Since $P$ generates $A$ as an algebra, so does $\Prim(A)$. Therefore, $\Phi$ is surjective. Since $A$ is a free $R$-module, so is $\Prim(A)$ by \Cref{lemma:submodule_free_is_free}. By the PBW theorem for $\mathbf{U}_R(\Prim(A))$ \cite[Chap. I, \S2, no.7, Theorem 1]{bourbaki1975lie}, since $\Prim(A)$ is a free $R$-module the $R$-module $\mathbf{U}_R(\Prim(A))$ is also free over $R$. In particular, it is torsion free. By tensoring with the fraction field $K=\Frac(R)$, we obtain a surjective morphism $\Phi_K\colon\mathbf{U}_R(\Prim(A))\otimes_R K\rightarrow A\otimes_RK$. Since enveloping algebras commute with extensions of the base-ring \cite[Chap. 1, \S2, no.9]{bourbaki1975lie}, we have $\mathbf{U}_R(\Prim(A))\otimes_RK\cong \mathbf{U}_K(\Prim(A)\otimes_RK)$. Moreover, by freeness of $A$ and $\Prim(A)$, the morphism $\Prim(A)\otimes_RK\rightarrow A_K\coloneqq A\otimes K$ is injective and its image consists only of primitive elements, and so it factors through $\Prim(A\otimes K)$. The morphism $\Phi_K$ therefore factors as
 \[
  \mathbf{U}_K(\Prim(A)\otimes K)\rightarrow\mathbf{U}_K(\Prim(A\otimes K))\rightarrow A_K\,.
 \]
By the Milnor--Moore theorem for $A_K$, the second arrow is an isomorphism. The first arrow is injective, and the composition is $\Phi_K$ and is surjective. This implies that the composition is an isomorphism and also $\Prim(A)\otimes_R K\cong\Prim(A\otimes_RK)$.

Let $L=\ker(\Phi)$. Then, $L\otimes_RK=0$ and $L\subset \mathbf{U}_R(\Prim(A))$ is torsion-free by the PBW theorem. Therefore, $L=0$ and so $\Phi$ is an isomorphism.
\end{proof}

\section{Statements and Declarations}
The author has no competing interests to declare.

\printbibliography
\end{document}